\documentclass[11pt,reqno]{amsart}

\usepackage{amsfonts,amssymb,amsmath,mathrsfs,graphicx}
\usepackage{float}
\usepackage{epsfig}
\usepackage{graphicx}
\usepackage{tikz}
\usepackage{color}
\usetikzlibrary{calc}
\usetikzlibrary{decorations.pathmorphing}
\usepackage{hyperref}

\def\d{{\partial}}
\def\R{{\mathbf R}}
\DeclareMathOperator{\diver}{div}
\DeclareMathOperator{\RE}{Re}
\DeclareMathOperator{\IM}{Im}

\makeatletter

\title[]{The Wigner-Lohe model for quantum synchronization and its emergent dynamics}

\author[Paolo Antonelli]{Paolo Antonelli}
\address[Paolo Antonelli]{\newline Gran Sasso Science Institute, \newline viale F. Crispi, 7,  67100 L'Aquila, Italy}
\email{paolo.antonelli@gssi.infn.it}

\author[Seung-Yeal Ha]{Seung-Yeal Ha}
\address[Seung-Yeal Ha]{\newline Department of Mathematical Sciences and Research Institute of Mathematics \newline  Seoul National University,
 Seoul 151-747, Korea (Republic of) \newline
 Korea Institute for Advanced Study, Hoegiro 87, Seoul, 130-722, Korea (Republic of)}
\email{syha@snu.ac.kr}

\author[Dohyun Kim]{Dohyun Kim}
\address[Dohyun Kim]{\newline Department of Mathematical Sciences \newline  Seoul National University, Seoul 151-747, Korea (Republic of)}
\email{dohyunkim@snu.ac.kr}

\author[Pierangelo Marcati]{Pierangelo Marcati}
\address[Pierangelo Marcati]{\newline Department of Information Engineering, Computer Science and Matheamtics,  \newline  University of L'Aquila and Gran Sasso Science Institute, 67100 L'Aquila, Italy}
\email{pierangelo.marcati@univaq.it}

\date\today

\begin{document}

\newtheorem{theorem}{Theorem}[section]
\newtheorem{lemma}{Lemma}[section]
\newtheorem{corollary}{Corollary}[section]
\newtheorem{proposition}{Proposition}[section]
\newtheorem{remark}{Remark}[section]
\newtheorem{definition}{Definition}[section]
\newtheorem{example}{Example}[section]

\renewcommand{\theequation}{\thesection.\arabic{equation}}
\renewcommand{\thetheorem}{\thesection.\arabic{theorem}}
\renewcommand{\thelemma}{\thesection.\arabic{lemma}}
\newcommand{\bbr}{\mathbb R}
\newcommand{\bbz}{\mathbb Z}
\newcommand{\bbn}{\mathbb N}
\newcommand{\bbs}{\mathbb S}
\newcommand{\bbt}{\mathbb T}
\newcommand{\bbp}{\mathbb P}
\newcommand{\bbc}{\mathbb C}
\newcommand{\ddiv}{\textrm{div}}
\newcommand{\bn}{\bf n}
\newcommand{\rr}[1]{\rho_{{#1}}}
\newcommand{\thh}{\theta}
\def\charf {\mbox{{\text 1}\kern-.24em {\text l}}}
\renewcommand{\arraystretch}{1.5}

\thanks{\textbf{Acknowledgment.} The work of P. Antonelli and P. Marcati is partially supported by PRIN grant 2015YCJY3A\_003 and by G.N.A.M.P.A. (I.N.d.A.M.) , The work of S.-Y. Ha and D. Kim are partially supported by the National Research Foundation of Korea (NRF2014R1A2A205002096).}

\begin{abstract}
We present the Wigner-Lohe model for quantum synchronization which can be derived from the Schr\"{o}dinger-Lohe model using the Wigner formalism. For identical one-body potentials, we provide a priori sufficient framework leading the complete synchronization, in which $L^2$-distances between all wave functions tend to zero asymptotically. 
\end{abstract}

\maketitle

\tableofcontents

%%%%%%%%%%%%%%%%%%%%%%%%%%%%%%%%%%%%%%%%%%%%%%%%%%%%%%%%%%%%%%%%%%%%%%%%%%%%%%%%%%%%%%%%%%%%%%%%%%%%%%%
%
%  Section
%
%%%%%%%%%%%%%%%%%%%%%%%%%%%%%%%%%%%%%%%%%%%%%%%%%%%%%%%%%%%%%%%%%%%%%%%%%%%%%%%%%%%%%%%%%%%%%%%%%%%%%%

\section{Introduction}
\setcounter{equation}{0}
Synchronization represents a phenomenon in which rhythms of weakly coupled oscillators are adjusted to the 
common frequency due to their weak interactions. It is often observed in many complex systems, e.g.,
 the flashing of fireflies, clapping of hands in a concert hall, and heartbeat regulation by pacemaker cells, etc.,
\cite{A-B, B-S, B-B,  Pe, St}. However, rigorous mathematical treatment of such collective phenomena were begun only several decades ago
by two scientists Winfree \cite{Wi2} and Kuramoto \cite{Ku1, Ku2}. 
For the mathematical modeling of synchronization, they adopted a continuous dynamical
 system approach based on their heuristic and intuitive arguments. 
 In this paper, we are mainly interested in quantum Lohe oscillators with all-to-all couplings
under one-body potential. To fix the idea, consider a classical complete network consisting of $N$ nodes, where each pair of nodes is
connected with an equal capacity which is assumed to be unity. We also assume that quantum Lohe oscillators with the same unit mass are positioned on the nodes of 
the underlying complete network. To avoid unnecessary physical complexity, we ignore entanglement and decoherence effects inherent to the quantum many-body systems. For a better 
physical modeling, such genuine quantum effects need to be taken into account. \newline

Let $\psi_i = \psi_i(t, x)$ be the wave function of the $i$-th Lohe oscillator on the spatial domain $\bbr^d$. 
Then, the dynamics of Lohe oscillators with unit mass is governed by the Schr\"{o}dinger-Lohe (S-L) model: for $(t, x)\in\mathbb R_+\times\mathbb R^d$.
\begin{align}
\begin{aligned} \label{S-L}
{\mathrm i}  \partial_t \psi_i = -\frac{1}{2} \Delta \psi_i +  V_i \psi_i  + \frac{{\mathrm i} K}{2N} \sum_{k=1}^{N}
\Big( \frac{ \|\psi_i\| \psi_k}{\|\psi_k \|}  -  \frac{\langle \psi_k, \psi_i \rangle \psi_i }{\|\psi_i\| \|\psi_k\|} \Big), \quad 1 \leq i \leq N,
\end{aligned}
\end{align}
where $\| \cdot \| := \|\cdot\|_{L^2}$ and $\langle \cdot, \cdot \rangle$ are the standard $L^2$ norm and an inner product on $\bbr^d$, and $V_i= V_{i}(x)$ and $K$ correspond to
the one-body potential and nonnegative coupling strength, respectively.
The  S--L model \eqref{S-L} was first introduced by Australian physicist Max Lohe \cite{Lo-1} several years ago
as an infinite state generalization of the Lohe matrix model \cite{Lo}. As discussed in \cite{Lo-1, Lo}, quantum synchronization has received much attention
from the quantum optics community because of its possible applications in
quantum computing and quantum information \cite{D-W-K, G-C, G-G, Ki, M-K, V-B, Z-S1, Z-S2}. The emergent dynamics of the S-L system \eqref{S-L} has been 
partially treated in \cite{C-C-H, C-H} for some restricted class of initial data and a large coupling strength.  
Recently, a new approach based on the finite-dimensional reduction has been proposed in \cite{AM, H-H} which significantly improve  the previous results \cite{C-C-H, C-H} by the Lyapunov functional approach. 
However, a complete resolution of the synchronization problem for \eqref{S-L} is still far from complete.
\newline

Our main purpose of this paper is to present a quantum kinetic analogue of the S-L model \eqref{S-L} and study its emergent dynamics. The study on the quantum kinetic model for the Schr\"{o}dinger equation dates back to Wigner's paper \cite{Wig}, in which Wigner considered the quantum mechanical motion of  a large ensemble of electrons in a vacuum under the action of the Coulomb force generated by the charge of the electrons. For the modeling of large ensemble, he introduced a quasi one-particle distribution function, so called the Wigner function and showed that it satisfies the quantum Liouville equation \cite{B-I-Z, B-M, Il, I-Z-L,Mak,Zw}.\newline  

Before we briefly describe our main results, we first recall the Wigner transform of wave function on $\bbr^d.$ For more basic facts on Wigner transforms we refer the reader to \cite{GMMP, GMP}.
\begin{definition}
For any two wave functions $\psi, \phi\in L^2$, we define the Wigner transform
\begin{equation*}
w[\psi, \phi](x, p)=\frac{1}{(2\pi)^d}\int_{\bbr^d} e^{iy\cdot p}\bar\psi\left(x+\frac{y}{2}\right)\phi\left(x-\frac{y}{2}\right)\,dy.
\end{equation*}
If we choose $\psi=\phi$, then we write $w[\psi] :=w[\psi, \psi]$.
\end{definition}
In order to shorten the formulas, we are going to introduce the following notation: if $\psi_j$, $j=1, \dotsc, N$ is the solution to the S-L system \eqref{S-L}, then we write 
\[ w_j:=w[\psi_j], \quad  w_{jk}:=w[\psi_j, \psi_k], \quad w_{jk}^+:=\RE w_{jk} \quad \mbox{and} \quad w_{jk}^-:=\IM w_{jk}. \]

Our main results of this paper are as follows. First, we show that the Wigner transforms $w_i$ and $w_{ij}^{\pm}$ satisfies a coupled non-local system:
\begin{equation} \label{W-L-0}
\begin{cases}
\displaystyle \partial_t w_j+ p \cdot \nabla_x w_j +\Theta[V](w_j) = \frac{K}{N} \sum_{k=1}^N \Big[ w^+_{jk} - \Big(\int w^+_{jk} dp dx \Big) w_j\Big], \quad x \in \bbr^d,~t > 0, \\
\displaystyle \partial_t w^+_{jk} + p \cdot \nabla_x w^+_{jk} +\Theta[V](w^+_{jk}) \\
\displaystyle \hspace{1cm} = \frac{K}{2N} \sum_{\ell=1}^N \Big[w^+_{j\ell}+w^+_{\ell k}- \Big(\int(w^+_{j\ell}+w^+_{\ell k}) dp dx  \Big)w^+_{jk} + 
\Big(\int(w^-_{j\ell}+w^-_{\ell k}) dp dx  \Big) w^-_{jk} \Big], \\
\displaystyle \partial_t w^-_{jk} + p \cdot \nabla_x w^-_{jk} +\Theta[V](w^-_{jk}) \\
\displaystyle \hspace{1cm} = \frac{K}{2N} \sum_{\ell=1}^N \Big[w^-_{j\ell}+w^-_{\ell k}- \Big(\int(w^+_{j\ell}+w^+_{\ell k}) dp dx  \Big) w^-_{jk} + 
\Big(\int(w^-_{j\ell}+w^-_{\ell k}) dp dx  \Big) w^+_{jk} \Big].
 \end{cases}
\end{equation}
Second, we derive a sufficient condition for the complete synchronization of the coupled system \eqref{W-L-0}. 
Finally, we also investigate the hydrodynamic formulation for the Schr\"odinger-Lohe system \eqref{S-L} and derive synchronization estimates in some special cases.
\newline

The rest of this paper is organized as follows. In Section 2, we present the Schr\"{o}dinger-Lohe model for quantum synchronization and discuss  previous works on the complete synchronization of the S-L model. In Section 3, we derive our augmented Wigner-Lohe model from the S-L model using the Wigner transform. In Section 4, we present a priori complete synchronization estimates for some restricted class of initial data. In Section 5, we also discuss a hydrodynamic model which can be obtained from the S-L model for two-oscillator case.

\section{Preliminaries}
\setcounter{equation}{0}
In this section, we briefly present the Schr\"{o}dinger-Lohe (S-L) model for Lohe synchronization, and review earlier results on the synchronization problem for the S-L model.

\subsection{The Schr\"{o}dinger-Lohe model} As a phenomenological model for  the quantum synchronization generalizing classical Kuramoto synchronization, 
Lohe proposed a coupled Schr\"{o}dinger-type model in \cite{Lo-1}. For $(t, x) \in \bbr_+ \times \bbr^d$ and $1 \leq i \leq N,$
 \begin{equation} \label{B-1}
{\mathrm i} \partial_t \psi_i =  -\frac{1}{2} \Delta \psi_i +  V_i \psi_i  + \frac{{\mathrm i} K}{2N} \sum_{k=1}^{N}
\Big( \frac{ \|\psi_i\| \psi_k}{\|\psi_k \|}  -  \frac{\langle \psi_k, \psi_i \rangle \psi_i }{\|\psi_i\| \|\psi_k\|} \Big),
\end{equation}
where we normalized $\hbar = 1$ and $m = 1$. 
 %\begin{equation} \label{S-L-1}
%{\mathrm i} \hbar \partial_t \psi_i  = \Big(H_i + H_i^{\mbox{int}} \Big) \psi_i,
%\end{equation}
%where $H_i$ and $H_i^{\mbox{int}}$ are free Hamiltonian and interaction Hamiltonian operators corresponding to the $i$-th oscillator: \begin{align} \begin{aligned} \label{S-L-1-1}
%H_i \psi_i &:=   -\frac{\hbar^2}{2} \Delta \psi_i + V_{i}(x) \psi_i, \\ H_i^{\mbox{int}} \psi_i &:= \frac{{\mathrm i} \hbar K}{2N} \sum_{k=1}^{N} \Big(  {\hat \psi}_k {\hat \psi}^*_i   -  {\hat \psi}_i  {\hat \psi}^*_k \Big) \psi_i,
%\end{aligned} \end{align}
%where ${\hat \psi}_i = \frac{\psi_i}{||\psi_i||}$ is the normalized state of $\psi_i$, and $\psi^*_k$ is the costate  generated by the state $\psi_k$, i.e.,
% \[ \psi_k^*(\psi_i) := \langle \psi_k, \psi_i \rangle. \]
% Note that system \eqref{S-L-1} and \eqref{S-L-1-1} is homogeneous in the sense that it is invariant under the similarity transformation $ |\psi_i \rangle \mapsto \lambda |\psi_i \rangle,~\lambda > 0$.

\begin{lemma}\label{L2.1}
\emph{\cite{Lo-1}} Let $\Psi = (\psi_1, \cdots \psi_N)$ be a smooth solution to \eqref{B-1}  with initial data $\Psi_0 = (\psi^0_{1}, \cdots, \psi^0_{N})$. Then, the $L^2$ norm of $\psi_i$ is constant along the flow \eqref{B-1}:
 \[  \| \psi_i(t) \|= \| \psi^0_{i} \| ~~~\mbox{ for }~~~ t \geq 0,~ 1 \leq i \leq N. \]
\end{lemma}
In view of the previous lemma, from now on we will assume that $\|\psi^0_{i}\|=1$, $1\leq i\leq N$, so that system \eqref{B-1} becomes
\begin{equation}\label{SL}
{\mathrm i} \d_t\psi_j=-\frac12\Delta\psi_j+V_j\psi_j+ \frac{{\mathrm i}K}{2N}\sum_{k=1}^N\left(\psi_k-\langle\psi_k, \psi_j\rangle\psi_j\right).
\end{equation}
For the space-homogeneous case, we set the spatial domain to be a periodic domain $\bbt^d$ and choose a special choice of $V_i$:
\[ V_{i}(x) =  \Omega_i: \mbox{constant}, \qquad \psi_i(t, x) = \psi_i(t), \quad (t, x) \in\bbr_+\times \bbt^d. \]
system \eqref{B-1} can be reduced to the Kuramoto model which is a prototype model for classical synchronization. In this special setting, the  S-L model becomes
\begin{equation} \label{S-L-h}
{\mathrm i} \frac{ d  \psi_i} {d t}  =  \Omega_i \psi_i   + \frac{{\mathrm i}  K}{2N} \sum_{k=1}^{N}
\Big( \frac{ |\psi_i| \psi_k}{|\psi_k|}  -  \frac{\langle \psi_i, \psi_k \rangle \psi_i }{|\psi_i|  |\psi_k|} \Big).
\end{equation}
We next simply take the ansatz for $\psi_i$:
\begin{equation} \label{B-2}
 \psi_i := e^{-{\mathrm i} \theta_{i}}, \quad 1 \leq i \leq N
\end{equation}
and substitute this ansatz into \eqref{S-L-h} to obtain
\[ {\dot \theta}_i \psi_i = \Omega_i \psi_i + \frac{{\mathrm i} K}{2N} \sum_{k=1}^{N} \Big(\psi_k - e^{-{\mathrm i}(\theta_i - \theta_k)} \psi_i  \Big). \]
Then, we take the inner product of the above relation with $\psi_i$ and compare the real part of the resulting relation to get the Kuramoto model for
classical synchronization \cite{A-B, B-S, C-H-J-K, H-K-R, H-N-P1, H-N-P2}:
\begin{equation} \label{B-3}
  {\dot \theta}_i = \Omega_i + \frac{K}{N} \sum_{k=1}^{N} \sin (\theta_k - \theta_i).   
\end{equation}  
Thus, the S-L model can be viewed as a quantum generalization of the Kuramoto model. 

\subsection{Previous results} In this subsection, we briefly review the previous results \cite{C-C-H, C-H, H-H, AM} on the complete synchronization of the S-L model. For this, we first recall the definition of the complete synchronization as follows.
\begin{definition} \label{D2.1} Let $\Psi = (\psi_1, \cdots \psi_N)$ be a smooth solution to \eqref{B-1}  with initial data $\Psi^0 = (\psi_{1}^{0}, \cdots, \psi_{N}^{0})$. Then, the S-L model exhibits an asymptotic phase-locking if the following relations holds:
\begin{equation} \label{B-4}
 \exists~~\lim_{t \to \infty} \langle \psi_i, \psi_j \rangle = \alpha_{ij} \in \bbc.  
\end{equation} 
\end{definition}
\begin{remark} For the classical phase models such as \eqref{B-3}, asymptotic phase-locking is defined as the following condition:
\begin{equation} \label{B-5}
 \exists~~\lim_{t \to \infty} |\theta_i(t) - \theta_j(t)| = \theta_{ij}^{\infty}.
\end{equation}
Via the relation \eqref{B-2}, we can see that \eqref{B-4} and \eqref{B-5} are closely related. In fact, in \cite{C-H} for identical 
potentials $V_i = V_j$, the complete synchronization is defined as 
\begin{equation} \label{B-6}
\lim_{t \to \infty} \|\psi_i(t) - \psi_j(t) \| = 0, \quad 1 \leq i, j \leq N.
\end{equation}
Note that the condition \eqref{B-6} and normalization condition $||\psi_i|| = 1$ yield
\[ \lim_{t \to \infty} \langle \psi_i(t), \psi_j(t) \rangle = 1. \]
Thus, the condition \eqref{B-6} satisfies the condition \eqref{B-4}. Recently, in \cite{AM, H-H} the case with different one-body potentials was treated, at least for $N=2$. In this framework it is shown that, in some regimes, the limit in \eqref{B-4} is not $1$ but depends on the difference between the potentials. Hence the limit in \eqref{B-6} gives a non-zero constant. This is indeed the more general case, when the system \eqref{SL} exhibits complete frequency synchronization but not phase synchronization. For more details we address the reader to \cite{AM, H-H}.
\end{remark}
As mentioned in the Introduction, the S-L model was first considered in Lohe's work \cite{Lo-1} for the non-Abelian generalization of the Kuramoto model. However, the first rigorous mathematical studies of the S-L model were treated by the second author and his collaborators in \cite{C-H, H-H} 
in two different methodologies. The first methodology is to use $L^2$-diameters for $\{ \psi_i \}$ as a Lyapunov functional and derive a 
Gronwall type differential inequality to conclude the complete synchronization with $\alpha_{ij} = 1$. More precisely, we set 
\[ D(\Psi):= \max_{i,j} ||\psi_i - \psi_j||. \]
In \cite{C-H}, authors derived a differential inequality for the diameter $D(\Psi)$:
\[ \frac{d}{dt} D(\Psi) \leq K (D(\Psi)) \Big(D(\Psi)-\frac{1}{2} \Big), \quad t > 0. \]
This leads to an exponential synchronization of the \eqref{S-L}.
\begin{theorem}\label{T2.1}
\emph{\cite{C-H}}
Suppose that the coupling strength and initial data satisfy
\[  K>0, \quad V_i = V, \quad  \|\psi^0_{i}\|_{L^2} = 1, \quad 1 \leq i \leq N, \quad   D(\Psi^0) <\frac{1}{2}.  \]
Then, for any solution $\Psi= (\psi_1, \dots, \psi_N)$ to \eqref{S-L}, the diameter $D(\Psi)$ satisfies
\[ D(\Psi(t)) \leq  \frac{D(\Psi^0)}{D(\Psi^0) + (1-2 D(\Psi^0)) e^{K t}}, \quad t \geq 0. \]
\end{theorem}
\begin{remark} For distinct one-body potentials, we do not have an asymptotic phase-locking estimate for the S-L model yet, however 
in \cite{C-C-H}, for some restricted class of initial data and large coupling strength, a weaker concept of synchronization, namely practical synchronization 
estimates have been obtained:
\[  \lim_{K \to \infty} \limsup_{t \to \infty} \max_{i, j} || \psi_i - \psi_j || = 0. \]
On the other hand, at least in the two oscillator case, it is possible to improve considerably the practical synchronization result: indeed in \cite{AM, H-H} a complete picture of different regimes is shown, where the system \eqref{SL} exhibits complete synchronization or dephasing, i.e. time periodic orbits for the correlation function.
\end{remark}
Recently, an alternative approach to prove synchronization for the S-L model was proposed both in \cite{AM} and \cite{H-H}, by using a finite dimensional reduction. More precisely, in both papers the authors consider the correlations between the wave functions,
\begin{equation}\label{eq:corr_f}
z_{jk}(t):=\langle\psi_j, \psi_k\rangle(t)=r_{jk}(t)+is_{jk}(t),
\end{equation}
and they study their asymptotic behavior. Moreover, in \cite{AM} the introduction of the order parameter, defined in analogy with the classical Kuramoto model, allows to give a more general result.
\begin{theorem}\cite{AM}
Let $(\psi_1, \dotsc, \psi_N)\in\mathcal C(\bbr_+;L^2(\bbr^d))^N$ be the solution to \eqref{SL} with initial data $(\psi_1(0), \dotsc, \psi_N(0))=(\psi^0_{1}, \dotsc, \psi^0_{N})\in L^2(\bbr^d)^N$, and we assume that
\begin{equation} \label{New-1}
\sum_{k=1}^N\RE z_{jk}(0)>0,\quad\textrm{for any}\;j=1, \dotsc, N.
\end{equation}
Then we have
\begin{equation*}
|1-z_{jk}(t)|\lesssim e^{-Kt}, \quad\textrm{as}\;t\to\infty.
\end{equation*}
\end{theorem}

As we will see in the next sections, the same approach used in \cite{AM, H-H} will also be exploited to infer the synchronization results for the Wigner-Lohe model \eqref{W-L} and the hydrodynamical system \eqref{eq:QHD_Lohe}. More precisely, for the quantum hydrodynamic system \eqref{eq:QHD_Lohe} we are going to need also some synchronization estimates proved at the $H^1$ regularity level. Such estimates are proved in \cite{AM}.
\section{From Schr\"odinger-Lohe to Wigner-Lohe}
\setcounter{equation}{0}
In this section we present a kinetic quantum analogue ``{the Wigner-Lohe(W-L) model}" for the quantum synchronization, which can be derived from the Schr\"{o}dinger-Lohe(S-L) model \cite{Lo-1, Lo} via the Wigner transform. In this and following sections, we assume that all one-body potentials are identical 
\begin{equation*}
V_j(x)=V(x),\quad1\leq j\leq N.
\end{equation*}
Recall that for a given a solution $\psi$ to the free Schr\"odinger equation:
\begin{equation*}
i\d_t\psi=-\frac12\Delta\psi+V\psi,
\end{equation*}
then its Wigner transform $w=w[\psi]$ satisfies
\begin{equation*}
\d_tw+p\cdot\nabla_xw+\Theta[V]w=0,
\end{equation*}
where the operator $\Theta[V]$ is defined by
\begin{equation*}
\Theta[V](w)(x, p) :=-\frac{i}{(2\pi)^d}\int e^{i(p-p')\cdot y}\left(V\left(x+\frac{y}{2}\right)-V\left(x-\frac{y}{2}\right)\right)w(x, p')\,dp'dy.
\end{equation*}
Hence, to derive the Wigner-Lohe system \eqref{W-L-0} we just need to see how the nonlocal coupling in \eqref{SL} translates at the Wigner level. More precisely, let $\psi_j$ be a solution to \eqref{SL}, then by defining $w_j=w[\psi_j]$, we see that it satisfies
\begin{equation*}
\d_tw_j+p\cdot\nabla w_j+\Theta[V]w_j=R_j,
\end{equation*}
where the remainder term $R_j$ is given by
\begin{equation*}
\begin{aligned}
R_j=&\frac{1}{(2\pi)^d}\frac{K}{2N}\sum_{k=1}^N\int e^{ip\cdot y}\Big[ \bar\psi_k\left(t, x+\frac{y}{2}\right)\psi_j\left(t, x-\frac{y}{2}\right) \cr
&+\bar\psi_j\left(t, x+\frac{y}{2}\right)\psi_k\left(t, x-\frac{y}{2}\right)\Big]\,dy -\frac{1}{(2\pi)^d}\frac{K}{2N}\sum_{k=1}^N2r_{jk}w_j\\
=&\frac{K}{N}\sum_{k=1}^N\left(w_{jk}^+-r_{jk}w_j\right),
\end{aligned}
\end{equation*}
where $r_{jk}(t) :=\RE\langle\psi_j, \psi_k\rangle(t)=\int w_{jk}^+(t, x, p)\,dxdp$. Let us recall that this last equality comes from one of the basic properties of Wigner transforms, namely
\begin{equation*}
\int w[f, g](x, p)\,dxdp=\langle f, g\rangle,
\end{equation*}
for any $f, g \in L^2$. Resuming, the equation for $w_j$ is given by
\begin{equation*}
\d_tw_j+p\cdot\nabla_xw_j+\Theta[V]w_j=\frac{K}{N}\sum_{k=1}^N\left(w_{jk}^+-r_{jk}w_j\right).
\end{equation*}
We now need to derive the equation for $w_{jk}=w[\psi_j, \psi_k]$. Since the linear part in the S-L model \eqref{SL} is common for every wave functions (remember we chose identical potentials, $V_j\equiv V$), then the linear part in the Wigner equation for $w_{jk}$ will be exactly the same as for $w_j$. Consequently we also have
\begin{equation*}
\d_tw_{jk}+p\cdot\nabla_xw_{jk}+\Theta[V]w_{jk}=R_{jk},
\end{equation*}
where
\begin{equation*}\begin{aligned}
R_{jk}=&\frac{1}{(2\pi)^d}\frac{K}{2N}\sum_{\ell=1}^N\int e^{iy\cdot p}\left(\bar\psi_\ell\left(t, x+\frac{y}{2}\right)\psi_k\left(t, x+\frac{y}{2}\right)+\bar\psi_j\left(t, x+\frac{y}{2}\right)\psi_\ell\left(t, x+\frac{y}{2}\right)\right)\,dy\\
&-\frac{1}{(2\pi)^d}\frac{K}{2N}\sum_{\ell=1}^N\left(z_{j\ell}w_{jk}+z_{\ell k}w_{jk}\right).
\end{aligned}\end{equation*}
Let us recall that $z_{jk}$ is defined in \eqref{eq:corr_f} and we notice that $\overline{z_{jk}}=z_{kj}$. Hence we obtain
\begin{equation*}
R_{jk}=\frac{K}{2N}\sum_{\ell=1}^N\left(w_{j\ell}+w_{\ell k}-(z_{j\ell}+z_{\ell k})w_{jk}\right)
\end{equation*}
and the equation for $w_{jk}$ becomes
\begin{equation*}
\d_tw_{jk}+p\cdot\nabla w_{jk}+\Theta[V]w_{jk}=\frac{K}{2N}\sum_{\ell=1}^N\left(w_{j\ell}+w_{\ell k}-(z_{j\ell}+z_{\ell k})w_{jk}\right).
\end{equation*}
By using definitions for $w_{jk}^\pm$ and the linearity of operator $\Theta[V]$, we then obtain the Wigner-Lohe system
\begin{equation} \label{W-L}
\begin{cases}
\displaystyle \partial_t w_j+ p \cdot \nabla_x w_j \Theta[V]w_j = \frac{K}{N} \sum_{k=1}^N \Big( w^+_{jk} - r_{jk}w_j\Big),\\
\displaystyle \partial_t w^+_{jk} + p \cdot \nabla_x w^+_{jk} \Theta[V](w^+_{jk}) \\
\displaystyle \hspace{3cm} = \frac{K}{2N} \sum_{\ell=1}^N \Big[w^+_{j\ell}+w^+_{\ell k}-(r_{j\ell}+r_{\ell k})w^+_{jk} -i(s_{j\ell}+s_{\ell k}) w^-_{jk} \Big], \\
\displaystyle \partial_t w^-_{jk} + p \cdot \nabla_x w^-_{jk} \Theta[V](w^-_{jk}) \\
\displaystyle \hspace{3cm} = \frac{K}{2N} \sum_{\ell=1}^N \Big[w^-_{j\ell}+w^-_{\ell k}-(r_{j\ell}+R_{\ell k})w^-_{jk} -i(s_{j\ell}+s_{\ell k}) w^+_{jk} \Big], \\
 \end{cases}
\end{equation}
\section{Emergent dynamics of the W-L model for identical potentials}
\setcounter{equation}{0}
In this section, we focus on the Wigner-Lohe model with $N=2$. In this case, system \eqref{W-L} becomes
\begin{equation}\label{2-WL}
\left\{\begin{aligned}
&\d_tw_1+p\cdot\nabla_xw_1+\Theta[V]w_1=\frac{K}{2}(w_{12}^+-r_{12}w_1), \\
&\d_tw_2+p\cdot\nabla_x w_2+\Theta[V]w_2=\frac{K}{2}(w_{12}^+-r_{12}w_2), \\
&\d_tw_{12}+p\cdot\nabla_xw_{12}+\Theta[V]w_{12}=\frac{K}{4}(w_1+w_2-2z_{12}w_{12}),
\end{aligned}\right.\end{equation}
where we have 
\begin{equation}\label{eq:def}
w_{12}^+=\RE w_{12}, \quad z_{12}=z_{12}(t)=\int w_{12}\,dxdp, \quad r_{12}=\RE z_{12}. 
\end{equation}
Let us remark that the system \eqref{2-WL}, complemented with the definitions \eqref{eq:def} above, can be considered independently on the S-L system \eqref{SL}. For such a system we will prescribe initial data $w_1^0, w_2^0, w_{12}^0$ such that $w_1^0$ and $w_2^0$ are real valued, $\int w_1^0\,dxdp=\int w_2^0\,dxdp=1$, $w_{12}^0$ is complex valued, $|\int w_{12}^0\,dxdp|\leq1$.
Let us also notice that the last equation is complex valued, so that we don't split it into two coupled equations for $w_{12}^+$ and $w_{12}^-$ as in \eqref{W-L}.

Let us now prove the synchronization for \eqref{2-WL}. First of all we remark that, by integrating the last equation over the whole phase space, we find the following ODE
\begin{equation}\label{eq:ODE}
\dot z_{12}=\frac{K}{2}(1-z_{12}^2),
\end{equation}
for which it is straightforward to give its asymptotic behavior.
\begin{lemma}\label{prop:L2sync}
Let $z_{12}(0)\in\mathbb C$ be such that $|z_{12}(0)|\leq1$ and $z_{12}(0)\neq-1$, then the solution $z_{12}(t)$ to \eqref{eq:ODE} satisfies
\begin{equation*}
|1-z_{12}(t)|\lesssim e^{-Kt}.
\end{equation*}
\end{lemma}
\begin{proof}
By integrating \eqref{eq:ODE} we obtain
\begin{equation*}
z_{12}(t)=\frac{(1+z_{12}(0))e^{Kt}-(1-z_{12}(0))}{(1+z_{12}(0))e^{Kt}+(1-z_{12}(0))}.
\end{equation*}
\end{proof}
By using the Lemma above it is then possible to show the complete synchronization for the W-L model \eqref{2-WL}.
\begin{theorem}
Let $(w_1, w_2, w_{12})$ be a solution to \eqref{2-WL} with initial data $(w_1(0), w_2(0), w_{12}(0))=(w_1^0, w_2^0, w_{12}^0)$ such that
\begin{equation*}
\int w_1^0(x, p)\,dxdp=\int w_2^0(x, p)\,dxdp=1,
\end{equation*}
and 
\begin{equation*}
\Big|\int w_{12}^0(x, p)\,dxdp \Big|\leq1,\quad\int w_{12}^0(x, p)\,dxdp\neq-1.
\end{equation*}
Then we have
\begin{equation*}
\|w_1(t)-w_2(t)\|_{L^2}^2\leq e^{-Kt},\quad\textrm{as}\;t\to\infty.
\end{equation*}
\end{theorem}
\begin{proof}
It follows from \eqref{2-WL} that it is possible to write down the equation for the difference $w_1-w_2$,
\begin{equation*}
\d_t(w_1-w_2)+p\cdot\nabla_x(w_1-w_2)+\Theta[V](w_1-w_2)=-\frac{Kr_{12}}{2}(w_1-w_2).
\end{equation*}
By multiplying it by $2(w_1-w_2)$ and by integrating over the whole phase space we obtain
\begin{equation*}
\frac{d}{dt}\|w_1(t)-w_2(t)\|_{L^2}^2=-Kr_{12}(t)\|w_1(t)-w_2(t)\|_{L^2}^2.
\end{equation*}
By Lemma \ref{prop:L2sync} we know that $|1-r_{12}(t)|\lesssim e^{-Kt}$, hence by Gronwall's inequality we obtain the synchronization result.
\end{proof}

\section{Quantum Hydrodynamics}
\setcounter{equation}{0}
In this Section, we derive the hydrodynamic equations associated to the Schr\"odinger-Lohe model \eqref{SL}. Here we follow the approach developed in \cite{A-M1, A-M2, AMContMath} where a polar factorisation method is exploited in order to define the hydrodynamical quantities also in the vacuum region. In order to simplify the exposition we mainly focus on the case of two identical oscillators. In this case the Schr\"odinger-Lohe model reads
\begin{equation}\label{eq:SL2}
\left\{\begin{aligned}
{\mathrm i} \d_t\psi_1=&-\frac12\Delta\psi_1+V\psi_1+ \frac{{\mathrm i} K}{4}(\psi_2-\langle\psi_2, \psi_1\rangle\psi_1)\\
{\mathrm i} \d_t\psi_2=&-\frac12\Delta\psi_2+V\psi_2+ \frac{{\mathrm i}K}{4}(\psi_1-\langle\psi_1, \psi_2, \rangle\psi_2).
\end{aligned}\right.
\end{equation}
The case with $N$ non-identical oscillators can be treated similarly with obvious modifications, but the study of this special case will simplify substantially the exposition.

In order to derive the hydrodynamics associated to system \eqref{eq:SL2}, we first need to ensure that it is globally well-posed in $H^1(\R^d)$. 
This is indeed a strightforward application of the standard for nonlinear Schr\"odinger equations \cite{Caz}, see for example Proposition 2.1 in \cite{AM}.
Furthermore, let us notice that by defining $z_{12}(t)=\langle\psi_1, \psi_2\rangle(t)$, then this function satisfies the ODE \eqref{eq:ODE}. This is not surprising because the W-L model was indeed derived from \eqref{eq:SL2} and because of the property 
$\int w[\psi_1, \psi_2]\,dxdp=\langle\psi_1, \psi_2\rangle$. This implies that, under the same assumptions of Lemma \ref{prop:L2sync}, in this case we also have
\begin{equation*}
|1-z_{12}(t)|\lesssim e^{-Kt}.
\end{equation*}
However, this synchronization result is too weak to be exploited for quantum hydrodynamic system derived from \eqref{eq:SL2}. Indeed, as we already remarked above, the natural space for the hydrodynamics is the finite energy space, namely $H^1$ for the wave functions. Hence we need to improve the result in the space of energy. Here we will make use of Theorem 4.5 in \cite{AM}, where we address the reader for more general results in this direction.

Let us now consider the solution $(\psi_1, \psi_2)\in\mathcal C(\bbr_+;H^1)$ to system \eqref{eq:SL2}, given by Proposition 2.1 in \cite{AM}. To derive the hydrodynamic system associated with \eqref{eq:SL2}, we first define the mass densities, namely $\rho_1=|\psi_1|^2$ and $\rho_2=|\psi_2|^2$.
By differentiating those quantities with respect to time and by using the equations above, we obtain
\begin{equation*}
\left\{\begin{aligned}
\d_t\rho_1+\diver J_1=&\frac{K}{2}(\rho_{12}-r_{12}\rho_1),\\
\d_t\rho_2+\diver J_2=&\frac{K}{2}(\rho_{12}-r_{12}\rho_2),
\end{aligned}\right.
\end{equation*}
where the associated current densities are respectively given by 
\[ J_1 :=\IM(\bar\psi_1\nabla\psi_1), \quad  J_2 :=\IM(\bar\psi_2\nabla\psi_2). \]
 Furthermore, in the equation for the mass density we also find the interaction term 
$\rho_{12}=\RE(\bar\psi_1\psi_2)$, so that $r_{12}=\RE\langle\psi_1, \psi_2\rangle=\int\rho_{12}\,dx$.

Let us notice that $\rho_{12}$ is not a mass density, since in general it can also be negative. By using those definitions we can derive the evolution equations for the current densities $J_1$ and $J_2$. For instance, by differentiating $J_1$ with respect to time we find that
\begin{equation*}
\d_tJ_1+\diver(\RE(\nabla\bar\psi_1\otimes\nabla\psi_1))+\rho_1\nabla V=\frac14\nabla\Delta\rho_1+\frac{K}{2}(J_{12}-r_{12}J_1),
\end{equation*}
where the new interaction term here is given by
\begin{equation*}
J_{12}=\frac12\IM(\bar\psi_1\nabla\psi_2+\bar\psi_2\nabla\psi_1).
\end{equation*}
Next, we use the polar factorisation Lemma in \cite{A-M1, A-M2} to infer that, for $\psi_1\in H^1(\bbr^d)$, we have
\begin{equation*}
\RE(\nabla\bar\psi_1\otimes\nabla\psi_1)=\nabla\sqrt{\rho_1}\otimes\nabla\sqrt{\rho_1}+\Lambda_1\otimes\Lambda_1,\qquad\textrm{a.e. in }\bbr^d,
\end{equation*}
where $\sqrt{\rho_1}=|\psi_1|, \Lambda_1=\IM(\bar\phi_1\nabla\psi_1)$, $\phi_1$ is the polar factor for the wave function $\psi_1$ and we have $\sqrt{\rho_1}\Lambda_1=J_1$, see \cite{A-M1, A-M2, AMContMath} for more details on the polar factorisation.
In this way we can write down the following equation for the current density $J_1$:
\begin{equation*}
\d_tJ_1+\diver\left(\frac{J_1\otimes J_1}{\rho_1}\right)+\rho_1\nabla V=\frac12\rho_1\nabla\left(\frac{\Delta\sqrt{\rho_1}}{\sqrt{\rho_1}}\right)+\frac{K}{2}(J_{12}-r_{12}J_1).
\end{equation*}
By using the equation for $\psi_2$ we obtain an analogous equation for $J_2$:
\begin{equation*}
\d_tJ_2+\diver\left(\frac{J_2\otimes J_2}{\rho_2}\right)+\rho_2\nabla V=\frac12\rho_2\nabla\left(\frac{\Delta\sqrt{\rho_2}}{\sqrt{\rho_2}}\right)+\frac{K}{2}(J_{12}-r_{12}J_2).
\end{equation*}
Resuming, by defining the hydrodynamical quantities $\rho_1, J_1, \rho_2, J_2$ associated to $\psi_1, \psi_2$, respectively, we can derive the following system:
\begin{equation*}
\left\{\begin{aligned}
&\d_t\rho_1+\diver J_1=\frac{K}{2}(\rho_{12}-r_{12}\rho_1),\\
&\d_t\rho_2+\diver J_2=\frac{K}{2}(\rho_{12}-r_{12}\rho_2),\\
&\d_tJ_1+\diver\left(\frac{J_1\otimes J_1}{\rho_1}\right)+\rho_1\nabla V=\frac12\rho_1\nabla\left(\frac{\Delta\sqrt{\rho_1}}{\sqrt{\rho_1}}\right)+\frac{K}{2}(J_{12}-r_{12}J_1),\\
&\d_tJ_2+\diver\left(\frac{J_2\otimes J_2}{\rho_2}\right)+\rho_2\nabla V=\frac12\rho_2\nabla\left(\frac{\Delta\sqrt{\rho_2}}{\sqrt{\rho_2}}\right)+\frac{K}{2}(J_{12}-r_{12}J_2).
\end{aligned}\right.
\end{equation*}
Note that the above hydrodynamical system is not closed, as we need to derive also the evolution equations for the quantities $\rho_{12}, J_{12}$. However, it is quite troublesome to derive a hydrodynamical equation for the quantity $J_{12}$. For this reason we consider the following auxiliary variables
\begin{equation*}
\rho_d :=|\psi_1-\psi_2|^2,\quad J_d :=\IM((\overline{\psi_1-\psi_2})\nabla(\psi_1-\psi_2)).
\end{equation*}
By using those variables, it is straightforward to derive their dynamical equations,

\begin{equation*}
\begin{aligned}
\d_t\rho_d=&2\RE\left\{(\overline{\psi_1-\psi_2})\left(\frac{i}{2}\Delta(\psi_1-\psi_2)-iV(\psi_1-\psi_2)+\frac{K}{4}(\psi_2-\psi_1-\langle\psi_2, \psi_1\rangle\psi_1+\langle\psi_1, \psi_2\rangle\psi_2)\right)\right\}\\
=&-\diver J_d+\frac{K}{2}\RE\left\{(\overline{\psi_1-\psi_2})\left((1+\langle\psi_1, \psi_2\rangle)(\psi_2-\psi_1)+2i\IM\langle\psi_1, \psi_2\rangle\psi_1\right)\right\}\\
=&-\diver J_d-\frac{K}{2}(1+r_{12})\rho_d+Ks_{12}\sigma_{12},
\end{aligned}
\end{equation*}
where we denoted $\sigma_{12}=\IM(\bar\psi_1\psi_2)$, so that
\[ s_{12}=\IM\langle\psi_1, \psi_2\rangle=\int_{\bbr^d} \sigma_{12}\,dx. \]
Define $\psi_d :=\psi_1-\psi_2$, then by following some similar calculations as before we find out
\begin{equation*}\begin{aligned}
\d_tJ_d=&-\frac12\RE(\Delta\bar\psi_d\nabla\psi_d-\bar\psi_d\nabla\Delta\psi_d)-\rho_d\nabla V\\
&+\frac{K}{4}\left[-2J_d+\IM\left(\langle\psi_2, \psi_1\rangle(-\bar\psi_1\nabla\psi_d+\bar\psi_d\nabla\psi_2)+\langle\psi_1, \psi_2\rangle(\bar\psi_2\nabla\psi_d-\bar\psi_d\nabla\psi_1)\right)\right].
\end{aligned}\end{equation*}
After some simple algebra, we obtain that
\begin{equation*}
\IM\left(\langle\psi_2, \psi_1\rangle(-\bar\psi_1\nabla\psi_d+\bar\psi_d\nabla\psi_2)+\langle\psi_1, \psi_2\rangle(\bar\psi_2\nabla\psi_d-\bar\psi_d\nabla\psi_1)\right)
=-2r_{12}J_d-4s_{12}G_{12},
\end{equation*}
where $G_{12} :=\frac12\RE(\bar\psi_2\nabla\psi_1-\bar\psi_1\nabla\psi_2)$. Hence the equation for $J_d$ is given by
\begin{equation*}
\d_tJ_d+\diver\left(\frac{J_d\otimes J_d}{\rho_d}\right)+\rho_d\nabla V=\frac12\rho_d\nabla\left(\frac{\Delta\sqrt{\rho_d}}{\sqrt{\rho_d}}\right)-\frac{K}{2}\left((1+r_{12})J_d+2s_{12}G_{12}\right).
\end{equation*}
Once again, to close the hydrodynamic equations, we still need to determine the evolution for $\sigma_{12}, G_{12}$. As before, the equation derived for $G_{12}$ would be too involved, for this reason we alternatively define $\psi_a=\psi_1-{\mathrm i}\psi_2$ and its hydrodynamical quantities $\rho_a=\frac12|\psi_a|^2$, $J_a=\frac12\IM(\bar\psi_a\nabla\psi_a)$.
If we write down the equation for $\psi_a$, 
\begin{equation*}
i\d_t\psi_a=-\frac12\Delta\psi_a+V\psi_a+\frac{K}{4}\left(i\psi_2+\psi_1-i\langle\psi_2, \psi_1\rangle\psi_1-\langle\psi_1, \psi_2\rangle\psi_2\right),
\end{equation*}
we can then infer the equations for $\rho_a$ and $J_a$. By proceeding as before with some straightforward but long calculations, we find out
\begin{equation*}
\begin{aligned}
&\d_t\rho_a+\diver J_a=\frac{K}{2}\left((1-s_{12})\rho_{12}-r_{12}\rho_a\right)\\
&\d_tJ_a+\diver\left(\frac{J_a\otimes J_a}{\rho_a}\right)+\rho_a\nabla V=\frac12\rho_a\nabla\left(\frac{\Delta\sqrt{\rho_a}}{\sqrt{\rho_a}}\right)+\frac{K}{2}\Big((1-s_{12})J_{12}-r_{12}J_a\Big).
\end{aligned}
\end{equation*}
We can now resume and write down the whole set of hydrodynamic equations associated to the Schr\"odinger-Lohe system \eqref{eq:SL2}:
\begin{equation}\label{eq:QHD_Lohe}\begin{aligned}
&\d_t\rho_1+\diver J_1=\frac{K}{2}(\rho_{12}-r_{12}\rho_1)\\
&\d_t\rho_2+\diver J_2=\frac{K}{2}(\rho_{12}-r_{12}\rho_2),\\
&\d_tJ_1+\diver\left(\frac{J_1\otimes J_1}{\rho_1}\right)+\rho_1\nabla V=\frac12\rho_1\nabla\left(\frac{\Delta\sqrt{\rho_1}}{\sqrt{\rho_1}}\right)+\frac{K}{2}(J_{12}-r_{12}J_1),\\
&\d_tJ_2+\diver\left(\frac{J_2\otimes J_2}{\rho_2}\right)+\rho_2\nabla V=\frac12\rho_2\nabla\left(\frac{\Delta\sqrt{\rho_2}}{\sqrt{\rho_2}}\right)+\frac{K}{2}(J_{12}-r_{12}J_2),\\
&\d_t\rho_d+\diver J_d=-\frac{K}{2}(1+r_{12})\rho_d+Ks_{12}\sigma_{12},\\
&\d_tJ_d+\diver\left(\frac{J_d\otimes J_d}{\rho_d}\right)+\rho_d\nabla V=\frac12\rho_d\nabla\left(\frac{\Delta\sqrt{\rho_d}}{\sqrt{\rho_d}}\right)-\frac{K}{2}\left((1+r_{12})J_d+2s_{12}G_{12}\right),\\
&\d_t\rho_a+\diver J_a=\frac{K}{2}\left((1-s_{12})\rho_{12}-r_{12}\rho_a\right)\\
&\d_tJ_a+\diver\left(\frac{J_a\otimes J_a}{\rho_a}\right)+\rho_a\nabla V=\frac12\rho_a\nabla\left(\frac{\Delta\sqrt{\rho_a}}{\sqrt{\rho_a}}\right)+\frac{K}{2}\left((1-s_{12})J_{12}-r_{12}J_a\right),
\end{aligned}\end{equation}
where 
\begin{eqnarray*}
&& \rho_{12} :=\frac12(\rho_1+\rho_2-\rho_d), \quad J_{12} :=\frac12(J_1+J_2-J_d), \cr
&& \sigma_{12} :=\rho_a -\frac{1}{2}(\rho_1+\rho_2), \quad G_{12} :=J_a-\frac{1}{2}(J_1+J_2). 
\end{eqnarray*}
By considering the system \eqref{eq:QHD_Lohe} above we can now prove the synchronization property. In view of the previosu synchronization results we expect that
\begin{equation*}
\lim_{t\to\infty}\left(\|\nabla\sqrt{\rho_1}-\nabla\sqrt{\rho_2}\|_{L^2}+\|\Lambda_1-\Lambda_2\|_{L^2}\right)=0
\end{equation*}
and furthermore
\begin{equation*}
\lim_{t\to\infty}\left(\|\nabla\sqrt{\rho_d}\|_{L^2}+\|\Lambda_d\|_{L^2}\right)=0.
\end{equation*}
To show the above properties we are going to use a synchronization result in $H^1$ for system \eqref{eq:SL2} given in \cite{AM}, which will be stated in the following Theorem. The result below actually holds in a more general case, see \cite{AM} for more details, however here we will state the synchronization property we are going to use for our system \eqref{eq:QHD_Lohe}.
\begin{theorem}\cite{AM}
Let $(\psi^0_{1}, \psi^0_{2})\in H^1$ be such that
\begin{equation*}
\langle\psi^0_{1}, \psi^0_{2}\rangle\neq-1.
\end{equation*}
Then, for the solution $(\psi_1, \psi_2)\in\mathcal C(\bbr_+;H^1)$ emanated from such initial data, we have
\begin{equation}\label{eq:H1sync}
\|\psi_1(t)-\psi_2(t)\|_{H^1}\lesssim e^{-Kt},\quad\textrm{as}\;t\to\infty.
\end{equation}
\end{theorem}
We apply now this result for the synchronization of system \eqref{eq:QHD_Lohe}, First of all, from \eqref{eq:H1sync} we then infer that
\begin{equation*}
\lim_{t\to\infty}\|\psi_d(t)\|_{H^1}=0.
\end{equation*}
This and the polar factorisation Lemma 3 in \cite{A-M2} then readily implies that
\begin{equation*}
\lim_{t\to\infty}\left(\|\nabla\sqrt{\rho_d}(t)\|_{L^2}+\|\Lambda_d(t)\|_{L^2}\right)=0.
\end{equation*}
Furthermore, from the fact that $\psi_d\to0$ in $H^1$ as $t\to\infty$ and the polar factorisation Lemma, again, we can also show that
\begin{equation*}
\lim_{t\to\infty}\left(\|\nabla\sqrt{\rho_1}(t)-\nabla\sqrt{\rho_2}(t)\|_{L^2}+\|\Lambda_1(t)-\Lambda_2(t)\|_{L^2}\right)=0.
\end{equation*}

\bibliographystyle{amsplain}

\end{document}